\documentclass[12pt]{article}
\usepackage{amsmath,amssymb,amsthm}
\numberwithin{equation}{section}
\usepackage{units}
\usepackage{color}
\usepackage[T1]{fontenc}
\usepackage[utf8]{inputenc}
\usepackage{authblk}
\usepackage{bm}
\usepackage{graphicx}

\usepackage{datetime}

\usepackage[colorlinks=true,
linkcolor=webgreen,
filecolor=webbrown,
citecolor=webgreen]{hyperref}

\definecolor{webgreen}{rgb}{0,.5,0}
\definecolor{webbrown}{rgb}{.6,0,0}

\usepackage[toc,page]{appendix}

\newtheorem{thm}{Theorem}
\newtheorem{theorem}[thm]{Theorem}
\newtheorem{lemma}{Lemma}

\newtheorem{corollary}[thm]{Corollary}

\title{Summation Identities Involving Padovan and Perrin Numbers}

\author[]{Kunle Adegoke \\\href{mailto:adegoke00@gmail.com}{\tt adegoke00@gmail.com}}
\affil{Department of Physics and Engineering Physics, \mbox{Obafemi Awolowo University}, 220005 Ile-Ife, Nigeria}

\begin{document}
\date{}
%\date{September 8, 2018}
%\currenttime

\maketitle

\begin{abstract}
\noindent Unlike in the case of Fibonacci and Lucas numbers, there is a paucity of literature dealing with summation identities involving the Padovan and Perrin numbers. In this paper, we derive various summation identities for these numbers, including binomial and double binomial identities. Our results derive from the rich algebraic properties exhibited by the zeros of the characteristic polynomial of the Padovan/Perrin sequence.

\end{abstract}
%\tableofcontents
% \listoffigures
%\chead{\thepage}
\section{Introduction}
The Padovan numbers, $P_n$, are defined by
\begin{equation}\label{eq.tl83ils}
P_n=P_{n-2}+P_{n-3}\;(n\ge 3) \,,\quad P_0=P_1=P_2=1\,,
\end{equation}
and the Perrin numbers, $Q_n$, by
\begin{equation}\label{eq.tl83ils}
Q_n=Q_{n-2}+Q_{n-3}\;(n\ge 3) \,,\quad Q_0=3,\,Q_1=0,\,Q_2=2\,.
\end{equation}

\bigskip

Both sequences $(P_n)$ and $(Q_n)$ can be extended to negative indices by writing the recurrence relations as $P_{n}=P_{n+3}-P_{n+1}$ and $Q_{n}=Q_{n+3}-Q_{n+1}$ and replacing $n$ with $-n$, thus obtaining
\[
P_{ - n}  = P_{ - n + 3}  - P_{ - n + 1},\quad Q_{ - n}  = Q_{ - n + 3}  - Q_{ - n + 1}\,.
\]
We have (see Theorem \ref{thm.c2lx8du})
\[
P_{-n}=P_{n-7}^2-P_{n-6}P_{n-8}
\]
and
\[
2Q_{-n}=Q_n^2-Q_{2n}\,.
\]
Following is a list of the first few Padovan and Perrin numbers:\\
\begin{tabular}{lllllllllllllllllllll}
\multicolumn{1}{c}{$n$} & \multicolumn{1}{c}{$ -7$} & \multicolumn{1}{c}{$ -6$} & \multicolumn{1}{c}{$ -5$} & \multicolumn{1}{c}{$ -4$} & \multicolumn{1}{c}{$ -3$} & \multicolumn{1}{c}{$ -2$} & \multicolumn{1}{c}{$ -1$} & \multicolumn{1}{c}{$ 0$} & \multicolumn{1}{c}{$ 1$} & \multicolumn{1}{c}{$ 2$} & \multicolumn{1}{c}{$ 3$} & \multicolumn{1}{c}{$ 4$} & \multicolumn{1}{c}{$ 5$} & \multicolumn{1}{c}{$ 6$} & \multicolumn{1}{c}{$ 7$} & \multicolumn{1}{c}{$ 8$} & \multicolumn{1}{c}{$ 9$} & \multicolumn{1}{c}{$ 10$} & \multicolumn{1}{c}{$ 11$} & \multicolumn{1}{c}{$ 12$} \\ 
\hline
\multicolumn{1}{c}{$P_n$} & \multicolumn{1}{c}{$ 1$} & \multicolumn{1}{c}{$ -1$} & \multicolumn{1}{c}{$ 1$} & \multicolumn{1}{c}{$ 0$} & \multicolumn{1}{c}{$ 0$} & \multicolumn{1}{c}{$ 1$} & \multicolumn{1}{c}{$ 0$} & \multicolumn{1}{c}{$ 1$} & \multicolumn{1}{c}{$ 1$} & \multicolumn{1}{c}{$ 1$} & \multicolumn{1}{c}{$ 2$} & \multicolumn{1}{c}{$ 2$} & \multicolumn{1}{c}{$ 3$} & \multicolumn{1}{c}{$ 4$} & \multicolumn{1}{c}{$ 5$} & \multicolumn{1}{c}{$ 7$} & \multicolumn{1}{c}{$ 9$} & \multicolumn{1}{c}{$ 12$} & \multicolumn{1}{c}{$ 16$} & \multicolumn{1}{c}{$ 21$} \\ 
\multicolumn{1}{c}{$Q_n$} & \multicolumn{1}{c}{$ -1$} & \multicolumn{1}{c}{$ -2$} & \multicolumn{1}{c}{$ 4$} & \multicolumn{1}{c}{$ -3$} & \multicolumn{1}{c}{$ 2$} & \multicolumn{1}{c}{$ 1$} & \multicolumn{1}{c}{$ -1$} & \multicolumn{1}{c}{$ 3$} & \multicolumn{1}{c}{$ 0$} & \multicolumn{1}{c}{$ 2$} & \multicolumn{1}{c}{$ 3$} & \multicolumn{1}{c}{$ 2$} & \multicolumn{1}{c}{$ 5$} & \multicolumn{1}{c}{$ 5$} & \multicolumn{1}{c}{$ 7$} & \multicolumn{1}{c}{$ 10$} & \multicolumn{1}{c}{$ 12$} & \multicolumn{1}{c}{$ 17$} & \multicolumn{1}{c}{$ 22$} & \multicolumn{1}{c}{$ 29$} \\ 
\end{tabular}

\medskip

Compared to the related Fibonacci and Lucas sequences, there is a dearth of literature on Padovan and Perrin sequences. We mention Shannon et.~al.~\cite{shannon97} and Yilmaz and Taskara~\cite{yilmaz}. Useful information is contained in the Wikipedia articles~\cite{wikipad,wikiper}, Mathworld articles~\cite{mathworld_padovan,mathworld_perrin} and the Mathpages~\cite{mathpages_perrin} article on these numbers. 
\medskip

The purpose of this paper is to present binomial summation identities such as
\[
\sum_{j = 0}^{\left\lfloor {n/2} \right\rfloor } {( - 1)^j \frac{n}{{n - j}}\binom {n-j}jP_{p + n - 3j} }  = ( - 1)^n (Q_nP_p-P_{n+p} )\,,
\]
and
\[
\sum_{j = 0}^{\left\lfloor {n/2} \right\rfloor } {( - 1)^j \binom {n-j}jP_{p + n - 3j} }  = ( - 1)^{n - 1}(P_{p+1}P_{n-3}-P_pP_{n-2}) \,.
\]
We will also derive double binomial summation identities such as
\[
\sum_{j = 0}^{\left\lfloor {n/2} \right\rfloor } {\sum_{k = 0}^{n - 2j} {( - 1)^{j + k} \frac{n}{{n - j}}\binom {n-j}j\binom {n-2j}kQ_p^{n - 2j - k} P_{q - pj + pk} } }  = Q_{pn} P_q  - P_{pn + q}\,.
\]
Finally, we will derive ordinary summation identities such as the sum of Padovan numbers and Perrin numbers with subscripts in arithmetic progression:
\[
\sum_{j = 0}^n {P_{pj + q} }  = \frac{{\left| {\begin{array}{*{20}c}
   {P_{pn + p + q}  - P_q } & {P_{p - 3} } & {P_{p - 4} }  \\
   {P_{pn + p + q + 1}  - P_{q + 1} } & {P_{p - 2}  - 1} & {P_{p - 3} }  \\
   {P_{pn + p + q - 1}  - P_{q - 1} } & {P_{p - 4} } & {P_{p - 5}  - 1}  \\
\end{array}} \right|}}{{\left| {\begin{array}{*{20}c}
   {P_{p - 2}  - 1} & {P_{p - 3} } & {P_{p - 4} }  \\
   {P_{p - 1} } & {P_{p - 2}  - 1} & {P_{p - 3} }  \\
   {P_{p - 3} } & {P_{p - 4} } & {P_{p - 5}  - 1}  \\
\end{array}} \right|}}\,,
\]
\[
\sum_{j = 0}^n {Q_{pj + q} }  = \frac{{\left| {\begin{array}{*{20}c}
   {Q_{pn + p + q}  - Q_q } & {P_{p - 3} } & {P_{p - 4} }  \\
   {Q_{pn + p + q + 1}  - Q_{q + 1} } & {P_{p - 2}  - 1} & {P_{p - 3} }  \\
   {Q_{pn + p + q - 1}  - Q_{q - 1} } & {P_{p - 4} } & {P_{p - 5}  - 1}  \\
\end{array}} \right|}}{{\left| {\begin{array}{*{20}c}
   {P_{p - 2}  - 1} & {P_{p - 3} } & {P_{p - 4} }  \\
   {P_{p - 1} } & {P_{p - 2}  - 1} & {P_{p - 3} }  \\
   {P_{p - 3} } & {P_{p - 4} } & {P_{p - 5}  - 1}  \\
\end{array}} \right|}}\,,
\]
and the generating function of Padovan numbers with subscripts in arithmetic progression:
\[
\sum_{j = 0}^\infty  {P_{pj + q} y^j }  = \frac{{\left| {\begin{array}{*{20}c}
   {P_q } & { - yP_{p - 3} } & { - yP_{p - 4} }  \\
   {P_{q + 1} } & { - yP_{p - 2}  + 1} & { - yP_{p - 3} }  \\
   {P_{q - 1} } & { - yP_{p - 4} } & { - yP_{p - 5}  + 1}  \\
\end{array}} \right|}}{{\left| {\begin{array}{*{20}c}
   { - yP_{p - 2}  + 1} & { - yP_{p - 3} } & { - yP_{p - 4} }  \\
   { - yP_{p - 1} } & { - yP_{p - 2}  + 1} & { - yP_{p - 3} }  \\
   { - yP_{p - 3} } & { - yP_{p - 4} } & { - yP_{p - 5}  + 1}  \\
\end{array}} \right|}}\,.
\]
Here and throughout this paper, $|\cdot|$ denotes matrix determinant.

\medskip

Throughout this paper, we denote by $\alpha$, $\beta$ and $\gamma$, the zeros of the characteristic polynomial, \mbox{$x^3-x-1$}, of the Padovan sequence. 
\subsection{Algebraic properties of $\alpha$, $\beta$ and $\gamma$}
Vieta's formulas give
\begin{equation}\label{eq.nj4qi47}
\alpha+\beta=-\gamma,\quad\alpha\beta=1/\gamma,\quad\alpha\beta+\alpha\gamma+\beta\gamma=-1\,,
\end{equation}
from which we also infer
\begin{equation}\label{eq.mo0lb2w}
\alpha^2+\beta^2=\gamma^2-2/\gamma=-\gamma^2+2\,,
\end{equation}
\begin{equation}\label{eq.uafik1h}
(\alpha-\beta)^2=1-3/\gamma=4-3\gamma^2\,,
\end{equation}
\begin{equation}
\alpha^2\beta+\beta^2\alpha=-1=\alpha\beta(\alpha+\beta)\,,
\end{equation}
\begin{equation}\label{eq.i7l9sx8}
(\alpha^2-\beta^2)^2=\gamma^2-3\gamma
\end{equation}
and
\begin{equation}\label{eq.wiif5om}
\alpha/\beta+\beta/\alpha=\gamma-1\,.
\end{equation}
The following result is readily established by mathematical induction.
\begin{lemma}\label{lem.mjwk6zf}
The following identities hold for integer $n$:
\makeatletter
\renewcommand\tagform@[1]{\maketag@@@{\ignorespaces#1\unskip\@@italiccorr}}
\begin{equation}\tag{P1}
\alpha^n=\alpha^2P_{n-4}+\alpha P_{n-3}+P_{n-5}\,,
\end{equation}
\[\tag{P1a}
\beta^n=\beta^2P_{n-4}+\beta P_{n-3}+P_{n-5}
\]
and
\[\tag{P1c}
\gamma^n=\gamma^2P_{n-4}+\gamma P_{n-3}+P_{n-5}\,.
\]
\makeatother
\end{lemma} 
Standard techniques for solving difference equations give
\begin{equation}\label{eq.m0eml7d}
Q_n=\alpha^n+\beta^n+\gamma^n\,,
\end{equation}
which we shall often employ in the useful form
\begin{equation}\label{eq.gtb313p}
\alpha^n+\beta^n=Q_n-\gamma^n\,.
\end{equation}
Adding the identities in Lemma \ref{lem.mjwk6zf}, we have
\[
Q_n=\alpha^n+\beta^n+\gamma^n=Q_2P_{n-4}+Q_1P_{n-3}+3P_{n-5}\,,
\]
which allows us to express the Perrin numbers in terms of the Padovan numbers:
\begin{equation}\label{eq.zdw7lv5}
Q_n=2P_{n-4}+3P_{n-5}\,.
\end{equation} 
From identity \eqref{eq.mo0lb2w} and Lemma \ref{lem.mjwk6zf}, we also have
\begin{equation}\label{eq.vg2t1wo}
\alpha^n+\beta^n=-\gamma^2P_{n-4}-\gamma P_{n-3}+2P_{n-2}\,.
\end{equation}
Lemma \ref{lem.mjwk6zf} gives
\begin{equation}\label{eq.vf37uhb}
\alpha^n-\beta^n=(\alpha^2-\beta^2)P_{n-4}+(\alpha-\beta)P_{n-3}\,;
\end{equation}
so that
\begin{equation}\label{eq.rn0hi2p}
\frac{\alpha^n-\beta^n}{\alpha-\beta}=-\gamma P_{n-4}+P_{n-3}\,.
\end{equation}
Squaring \eqref{eq.vf37uhb} and making use of \eqref{eq.uafik1h} and \eqref{eq.i7l9sx8} gives
\begin{equation}
(\alpha ^n  - \beta ^n )^2  = (P_{n - 4}^2  - 3P_{n - 3}^2 ) \gamma^2  - (3P_{n - 4}^2  + 2P_{n - 4} P_{n - 3} )\gamma  + 4P_{n - 3}^2  + 6P_{n - 3} P_{n - 4}\,. 
\end{equation}
Since $P_{n-4}=P_{n-2}-P_{n-5}$, identity \eqref{eq.zdw7lv5} can also be written as
\begin{equation}\label{eq.d16qg78}
Q_n=2P_{n-2}+P_{n-5}\,.
\end{equation}
From \eqref{eq.d16qg78} and Lemma \ref{lem.mjwk6zf} we have
\begin{lemma}\label{lem.pqsiybc}
The following identities hold for integer $n$:
\makeatletter
\renewcommand\tagform@[1]{\maketag@@@{\ignorespaces#1\unskip\@@italiccorr}}
\begin{equation}\tag{P1d}
2\alpha^{n+2}+\alpha^{n-1}=\alpha^2Q_{n}+\alpha Q_{n+1}+Q_{n-1}\,,
\end{equation}
\makeatother
\[
2\beta^{n+2}+\beta^{n-1}=\beta^2Q_{n}+\beta Q_{n+1}+Q_{n-1}
\]
and
\[
2\gamma^{n+2}+\gamma^{n-1}=\gamma^2Q_{n}+\gamma Q_{n+1}+Q_{n-1}\,.
\]

\end{lemma} 
Presently, we derive more algebraic properties of $\alpha$, $\beta$ and $\gamma$.

\medskip

Since $\alpha$ is a zero of \mbox{$x^3-x-1$}, we have
\begin{equation}\label{eq.i581x7l}
\alpha+1=\alpha^3\,;
\end{equation}
so that
\[
\alpha^3-\alpha=1\Rightarrow\alpha(\alpha-1)(\alpha+1)=1\,,
\]
from which we get
\begin{equation}\label{eq.kmklhor}
\alpha-1=1/\alpha^4,\quad \text{by \eqref{eq.i581x7l}}\,,
\end{equation}
and
\begin{equation}\label{eq.g937rmf}
\alpha^2-1=1/\alpha\,.
\end{equation}
We also write identity \eqref{eq.kmklhor} as
\begin{equation}
\alpha^4+1=\alpha^5\,.
\end{equation}
On account of identity \eqref{eq.kmklhor}, identity \eqref{eq.wiif5om} is
\begin{equation}\label{eq.uajx4zx}
\alpha/\beta+\beta/\alpha=1/\gamma^{4}\,,
\end{equation}
which also implies
\begin{equation}\label{eq.sclithe}
1/\alpha^2+1/\beta^2=1/\gamma^3\,.
\end{equation}
Addition and subtraction of \eqref{eq.i581x7l} and \eqref{eq.kmklhor} give
\begin{equation}\label{eq.vp1p6sm}
\alpha^7+1=2\alpha^5
\end{equation}
and
\begin{equation}\label{eq.ozggl72}
\alpha^7-1=2\alpha^4\,,
\end{equation}
while, upon multiplication, \eqref{eq.vp1p6sm} and \eqref{eq.ozggl72} give
\begin{equation}\label{eq.oqvb3np}
\alpha^{14}-1=4\alpha^{\,9}.
\end{equation}
Identities \eqref{eq.i581x7l} -- \eqref{eq.oqvb3np}, excluding \eqref{eq.uajx4zx} and \eqref{eq.sclithe}, can be collected into the following lemma.
\begin{lemma}
The following identities hold for integer $m$, for each set of values of $a$, $b$, $c$, $d$, $e$, $f$ and $g$ given in the appended table:
\begin{equation}\label{eq.zan2cs6}
a\alpha^{m+c}+b\alpha^{m+d}=f\alpha^{m+e}\,,
\end{equation}
\begin{equation}
a\beta^{m+c}+b\beta^{m+d}=f\beta^{m+e}
\end{equation}
and
\begin{equation}
a\gamma^{m+c}+b\gamma^{m+d}=f\gamma^{m+e}\,.
\end{equation}
\begin{tabular}{lllllllllllllllllllll}
\multicolumn{1}{c}{} & \multicolumn{1}{c}{} & \multicolumn{1}{c}{$a$} & \multicolumn{1}{c}{$b$} & \multicolumn{1}{c}{$c$} & \multicolumn{1}{c}{$d$} & \multicolumn{1}{c}{$e$} & \multicolumn{1}{c}{$f$} & \multicolumn{1}{c}{} & \multicolumn{1}{c}{} & \multicolumn{1}{c}{} & \multicolumn{1}{c}{} & \multicolumn{1}{c}{} & \multicolumn{1}{c}{} & \multicolumn{1}{c}{} & \multicolumn{1}{c}{$a$} & \multicolumn{1}{c}{$b$} & \multicolumn{1}{c}{$c$} & \multicolumn{1}{c}{$d$} & \multicolumn{1}{c}{$e$} & \multicolumn{1}{c}{$f$} \\ 
\cline{1-8}\cline{14-21}
\multicolumn{1}{c}{Set 1:} & \multicolumn{1}{c}{} & \multicolumn{1}{c}{$1$} & \multicolumn{1}{c}{$-1$} & \multicolumn{1}{c}{$1$} & \multicolumn{1}{c}{$-1$} & \multicolumn{1}{c}{$-2$} & \multicolumn{1}{c}{$1$} & \multicolumn{1}{c}{} & \multicolumn{1}{c}{} & \multicolumn{1}{c}{} & \multicolumn{1}{c}{} & \multicolumn{1}{c}{} & \multicolumn{1}{c}{Set 9:} & \multicolumn{1}{c}{} & \multicolumn{1}{c}{$1$} & \multicolumn{1}{c}{$1$} & \multicolumn{1}{c}{$1$} & \multicolumn{1}{c}{$-6$} & \multicolumn{1}{c}{$-1$} & \multicolumn{1}{c}{$2$} \\ 
\multicolumn{1}{c}{Set 2:} & \multicolumn{1}{c}{} & \multicolumn{1}{c}{$1$} & \multicolumn{1}{c}{$-1$} & \multicolumn{1}{c}{$1$} & \multicolumn{1}{c}{$-2$} & \multicolumn{1}{c}{$-1$} & \multicolumn{1}{c}{$1$} & \multicolumn{1}{c}{} & \multicolumn{1}{c}{} & \multicolumn{1}{c}{} & \multicolumn{1}{c}{} & \multicolumn{1}{c}{} & \multicolumn{1}{c}{Set 10:} & \multicolumn{1}{c}{} & \multicolumn{1}{c}{$2$} & \multicolumn{1}{c}{$1$} & \multicolumn{1}{c}{$2$} & \multicolumn{1}{c}{$-2$} & \multicolumn{1}{c}{$5$} & \multicolumn{1}{c}{$1$} \\ 
\multicolumn{1}{c}{Set 3:} & \multicolumn{1}{c}{} & \multicolumn{1}{c}{$1$} & \multicolumn{1}{c}{$1$} & \multicolumn{1}{c}{$-1$} & \multicolumn{1}{c}{$-2$} & \multicolumn{1}{c}{$1$} & \multicolumn{1}{c}{$1$} & \multicolumn{1}{c}{} & \multicolumn{1}{c}{} & \multicolumn{1}{c}{} & \multicolumn{1}{c}{} & \multicolumn{1}{c}{} & \multicolumn{1}{c}{Set 11:} & \multicolumn{1}{c}{} & \multicolumn{1}{c}{$1$} & \multicolumn{1}{c}{$-1$} & \multicolumn{1}{c}{$5$} & \multicolumn{1}{c}{$-2$} & \multicolumn{1}{c}{$2$} & \multicolumn{1}{c}{$2$} \\ 
\multicolumn{1}{c}{Set 4:} & \multicolumn{1}{c}{} & \multicolumn{1}{c}{$1$} & \multicolumn{1}{c}{$1$} & \multicolumn{1}{c}{$2$} & \multicolumn{1}{c}{$-2$} & \multicolumn{1}{c}{$3$} & \multicolumn{1}{c}{$1$} & \multicolumn{1}{c}{} & \multicolumn{1}{c}{} & \multicolumn{1}{c}{} & \multicolumn{1}{c}{} & \multicolumn{1}{c}{} & \multicolumn{1}{c}{Set 12:} & \multicolumn{1}{c}{} & \multicolumn{1}{c}{$1$} & \multicolumn{1}{c}{$-2$} & \multicolumn{1}{c}{$5$} & \multicolumn{1}{c}{$2$} & \multicolumn{1}{c}{$-2$} & \multicolumn{1}{c}{$1$} \\ 
\multicolumn{1}{c}{Set 5:} & \multicolumn{1}{c}{} & \multicolumn{1}{c}{$1$} & \multicolumn{1}{c}{$-1$} & \multicolumn{1}{c}{$3$} & \multicolumn{1}{c}{$2$} & \multicolumn{1}{c}{$-2$} & \multicolumn{1}{c}{$1$} & \multicolumn{1}{c}{} & \multicolumn{1}{c}{} & \multicolumn{1}{c}{} & \multicolumn{1}{c}{} & \multicolumn{1}{c}{} & \multicolumn{1}{c}{Set 13:} & \multicolumn{1}{c}{} & \multicolumn{1}{c}{$1$} & \multicolumn{1}{c}{$-1$} & \multicolumn{1}{c}{$7$} & \multicolumn{1}{c}{$-7$} & \multicolumn{1}{c}{$2$} & \multicolumn{1}{c}{$4$} \\ 
\multicolumn{1}{c}{Set 6:} & \multicolumn{1}{c}{} & \multicolumn{1}{c}{$1$} & \multicolumn{1}{c}{$-1$} & \multicolumn{1}{c}{$3$} & \multicolumn{1}{c}{$-2$} & \multicolumn{1}{c}{$2$} & \multicolumn{1}{c}{$1$} & \multicolumn{1}{c}{} & \multicolumn{1}{c}{} & \multicolumn{1}{c}{} & \multicolumn{1}{c}{} & \multicolumn{1}{c}{} & \multicolumn{1}{c}{Set 14:} & \multicolumn{1}{c}{} & \multicolumn{1}{c}{$1$} & \multicolumn{1}{c}{$-4$} & \multicolumn{1}{c}{$7$} & \multicolumn{1}{c}{$2$} & \multicolumn{1}{c}{$-7$} & \multicolumn{1}{c}{$1$} \\ 
\multicolumn{1}{c}{Set 7:} & \multicolumn{1}{c}{} & \multicolumn{1}{c}{$2$} & \multicolumn{1}{c}{$-1$} & \multicolumn{1}{c}{$-1$} & \multicolumn{1}{c}{$1$} & \multicolumn{1}{c}{$-6$} & \multicolumn{1}{c}{$1$} & \multicolumn{1}{c}{} & \multicolumn{1}{c}{} & \multicolumn{1}{c}{} & \multicolumn{1}{c}{} & \multicolumn{1}{c}{} & \multicolumn{1}{c}{Set 15:} & \multicolumn{1}{c}{} & \multicolumn{1}{c}{$1$} & \multicolumn{1}{c}{$4$} & \multicolumn{1}{c}{$-7$} & \multicolumn{1}{c}{$2$} & \multicolumn{1}{c}{$7$} & \multicolumn{1}{c}{$1$} \\ 
\multicolumn{1}{c}{Set 8:} & \multicolumn{1}{c}{} & \multicolumn{1}{c}{$2$} & \multicolumn{1}{c}{$-1$} & \multicolumn{1}{c}{$-1$} & \multicolumn{1}{c}{$-6$} & \multicolumn{1}{c}{$1$} & \multicolumn{1}{c}{$1$} & \multicolumn{1}{c}{} & \multicolumn{1}{c}{} & \multicolumn{1}{c}{} & \multicolumn{1}{c}{} & \multicolumn{1}{c}{} & \multicolumn{1}{c}{} & \multicolumn{1}{c}{} & \multicolumn{1}{c}{} & \multicolumn{1}{c}{} & \multicolumn{1}{c}{} & \multicolumn{1}{c}{} & \multicolumn{1}{c}{} & \multicolumn{1}{c}{} \\ 
\end{tabular}
\end{lemma}
The observation that if $a$, $b$, $c$, $d$, $e$ and $f$ are rational numbers and $\lambda$ and $\lambda^2$ are {\it linearly independent} irrational numbers, then $a\lambda^2+b\lambda+c=d\lambda^2+e\lambda+f$ if and only if $a=d$, $b=e$ and $c=f$ leads to the following properties.
\begin{lemma}
If $a$, $b$, $c$, $d$, $e$ and $f$ are rational numbers, then:
\makeatletter
\renewcommand\tagform@[1]{\maketag@@@{\ignorespaces#1\unskip\@@italiccorr}}
\begin{equation}\tag{P2}
a\alpha^2+b\alpha+c=d\alpha^2+e\alpha+f \iff \mbox{$a=d$, $b=e$ and $c=f$}\,,
\end{equation}
\begin{equation}\tag{P3}
a\beta^2+b\beta+c=d\beta^2+e\beta+f \iff \mbox{$a=d$, $b=e$ and $c=f$}\,,
\end{equation}
\begin{equation}\tag{P4}
a\gamma^2+b\gamma+c=d\gamma^2+e\gamma+f \iff \mbox{$a=d$, $b=e$ and $c=f$}\,,
\end{equation}
\begin{equation}\tag{P5}
\frac{{a\alpha ^2  + b\alpha  + c}}{{d\alpha ^2  + e\alpha  + f}} = \frac{{\left| {\begin{array}{*{20}c}
   a & e & d  \\
   b & {d + f} & e  \\
   c & d & f  \\
\end{array}} \right|}}{{\left| {\begin{array}{*{20}c}
   {d + f} & e & d  \\
   {d + e} & {d + f} & e  \\
   e & d & f  \\
\end{array}} \right|}}\alpha ^2  + \frac{{\left| {\begin{array}{*{20}c}
   {d + f} & a & d  \\
   {d + e} & b & e  \\
   e & c & f  \\
\end{array}} \right|}}{{\left| {\begin{array}{*{20}c}
   {d + f} & e & d  \\
   {d + e} & {d + f} & e  \\
   e & d & f  \\
\end{array}} \right|}}\alpha  + \frac{{\left| {\begin{array}{*{20}c}
   {d + f} & e & a  \\
   {d + e} & {d + f} & b  \\
   e & d & c  \\
\end{array}} \right|}}{{\left| {\begin{array}{*{20}c}
   {d + f} & e & d  \\
   {d + e} & {d + f} & e  \\
   e & d & f  \\
\end{array}} \right|}}\,,
\end{equation}
for $d$, $e$, $f$ not all zero; with similar expressions for $\beta$ and $\gamma$; and, in particular,
\begin{equation}\tag{P6}
\frac{1}{{d\alpha ^2  + e\alpha  + f}} = \frac{{e^2  - d^2  - fd}}{{\left| {\begin{array}{*{20}c}
   {d + f} & e & d  \\
   {d + e} & {d + f} & e  \\
   e & d & f  \\
\end{array}} \right|}}\alpha ^2  + \frac{{d^2  - ef}}{{\left| {\begin{array}{*{20}c}
   {d + f} & e & d  \\
   {d + e} & {d + f} & e  \\
   e & d & f  \\
\end{array}} \right|}}\alpha  + \frac{{d^2  + 2fd + f^2  - ed - e^2 }}{{\left| {\begin{array}{*{20}c}
   {d + f} & e & d  \\
   {d + e} & {d + f} & e  \\
   e & d & f  \\
\end{array}} \right|}}\,,
\end{equation}
\begin{equation}\tag{P7}
\frac{1}{{e\alpha  + f}} = \frac{{e^2 }}{{f^3  - e^2 f + e^3 }}\alpha ^2  - \frac{{ef}}{{f^3  - e^2 f + e^3 }}\alpha  + \frac{{f^2  - e^2 }}{{f^3  - e^2 f + e^3 }}
\end{equation}
and
\begin{equation}\tag{P8}
\frac{1}{{d\alpha ^2  + e\alpha }} = \frac{{e^2  - d^2 }}{{e^3  + d^3  - ed^2 }}\alpha ^2  + \frac{{d^2 }}{{e^3  + d^3  - ed^2 }}\alpha  + \frac{{d^2  - ed - e^2 }}{{e^3  + d^3  - ed^2 }}\,.
\end{equation}
\makeatother
\end{lemma}
\subsubsection{A note on notation}
Since any function $F(\alpha,a,b,\ldots)$ or $G(\beta,a,b,\ldots)$ or $H(\gamma,a,b,\ldots)$, where $a,b,\ldots$ are rational numbers, can be considered a vector with three components, we can write
\begin{equation}
F=\left(F\right)_{\alpha^2}\alpha^2+\left(F\right)_{\alpha}\alpha+\left(F\right)_{\alpha^0}\,,
\end{equation}
\begin{equation}
G=\left(G\right)_{\beta^2}\beta^2+\left(G\right)_{\beta}\beta+\left(G\right)_{\beta^0}
\end{equation}
and
\begin{equation}
H=\left(H\right)_{\gamma^2}\gamma^2+\left(H\right)_{\gamma}\gamma+\left(H\right)_{\gamma^0}\,;
\end{equation}
so that $\left(F\right)_{\alpha^j}$, $j\in\{0,1,2\}$, denotes the $\alpha^j$ component of $F$, etc. 
Thus, for example, from identities P1, P5 and P6, we can write
\begin{equation}\tag{P9}
\left(\alpha^n\right)_{\alpha^2}=P_{n-4},\quad\left(\alpha^n\right)_{\alpha}=P_{n-3},\quad\left(\alpha^n\right)_{\alpha^0}=P_{n-5}\,,
\end{equation}
\begin{equation}\tag{P10}
\left(\frac{{a\alpha ^2  + b\alpha  + c}}{{d\alpha ^2  + e\alpha  + f}}\right)_{\alpha^2} = \frac{{\left| {\begin{array}{*{20}c}
   a & e & d  \\
   b & {d + f} & e  \\
   c & d & f  \\
\end{array}} \right|}}{{\left| {\begin{array}{*{20}c}
   {d + f} & e & d  \\
   {d + e} & {d + f} & e  \\
   e & d & f  \\
\end{array}} \right|}}
\end{equation}
and
\begin{equation}\tag{P11}
\left(\frac{1}{{d\alpha ^2  + e\alpha  + f}}\right)_{\alpha^2}= \frac{{e^2  - d^2  - fd}}{{\left| {\begin{array}{*{20}c}
   {d + f} & e & d  \\
   {d + e} & {d + f} & e  \\
   e & d & f  \\
\end{array}} \right|}}\,.
\end{equation}
\begin{theorem}\label{thm.c2lx8du}
For all integers $n$, we have
\begin{equation}\label{eq.q0kf6d2}
P_{-n}=P_{n-7}^2-P_{n-6}P_{n-8}
\end{equation}
and
\begin{equation}\label{eq.yna0jdm}
2Q_{-n}=Q_n^2-Q_{2n}\,.
\end{equation}
\end{theorem}
\begin{proof}
We have
\[
\left(\alpha^{-n}\right)_{\alpha^2}=\left(\frac1{\alpha^n}\right)_{\alpha^2}\,,
\]
which by P1 and P2 gives
\[
P_{ - n - 4}  = \frac{{P_{n - 3}^2  - P_{n - 4}^2  - P_{n - 4} P_{n - 5} }}{{\left| {\begin{array}{*{20}c}
   {P_{n - 4}  + P_{n - 5} } & {P_{n - 3} } & {P_{n - 4} }  \\
   {P_{n - 3}  + P_{n - 4} } & {P_{n - 4}  + P_{n - 5} } & {P_{n - 3} }  \\
   {P_{n - 3} } & {P_{n - 4} } & {P_{n - 5} }  \\
\end{array}} \right|}}= \frac{{P_{n - 3}^2  - P_{n - 4}^2  - P_{n - 4} P_{n - 5} }}{{\left| {\begin{array}{*{20}c}
   {P_{n - 2} } & {P_{n - 3} } & {P_{n - 4} }  \\
   {P_{n - 1} } & {P_{n - 2} } & {P_{n - 3} }  \\
   {P_{n - 3} } & {P_{n - 4} } & {P_{n - 5} }  \\
\end{array}} \right|}}\,,
\]
in which the numerator factors into $P_{n-3}^2-P_{n-2}P_{n-4}$ while the determinant in the denominator evaluates to $1$ for all $n$; and hence identity \eqref{eq.q0kf6d2}.

\medskip

Squaring \eqref{eq.m0eml7d} gives
\[
\begin{split}
Q_n^2  &= \alpha ^{2n}  + \beta ^{2n}  + \gamma ^{2n}  + 2\beta ^n \gamma ^n  + 2\alpha ^n \gamma ^n  + 2\alpha ^n \beta ^n\\
&\qquad= \underbrace{\alpha ^{2n}  + \beta ^{2n}  + \gamma ^{2n}}_{Q_{2n}}  + \underbrace{2\alpha ^{ - n}  + 2\beta ^{ - n}  + 2\gamma ^{ - n}}_{2Q_{-n}}\,,
\end{split}
\]
from which \eqref{eq.yna0jdm} follows.
\end{proof}
\begin{theorem}
The following identities hold for integers $p$ and $n$:
\begin{equation}\label{eq.dw7915i}
Q_{ - n} P_{p} - P_{p-n} = Q_nP_{p + n} -P_{p + 2n}
\end{equation}
and
\begin{equation}\label{eq.oyj86q9}
P_pP_{-n-3}-P_{p+1}P_{-n-4}=P_{p+n+1}P_{n-4}-P_{p+n}P_{n-3}\,.
\end{equation}

\end{theorem}
\begin{proof}
Set $x=\alpha$ and $y=\beta$ in the identity
\[
x^{-n}+y^{-n}=(xy)^{-n}(x^n+y^n)\,,
\]
multiply through by $\gamma^{p+4}$ and make use of identity \eqref{eq.gtb313p} to obtain 
\[
Q_{-n}\gamma^{p+4}-\gamma^{p-n+4}=Q_n\gamma^{p+n+4}-\gamma^{p+2n+4}\,,
\]
from which identity \eqref{eq.dw7915i} follows upon use of the $\gamma$ version of property P9.

\medskip

Set $x=\alpha$ and $y=\beta$ in the identity
\[
\frac{{x^{ - n}  - y^{ - n} }}{{x - y}} =  - (xy)^{ - n} \frac{{x^n  - y^n }}{{x - y}}\,,
\]
multiply through by $\gamma^{p+4}$ and make use of identity \eqref{eq.rn0hi2p} to obtain
\[
\gamma ^{p + 4} P_{ - n - 3}  - \gamma ^{p + 5} P_{ - n - 4}  = \gamma ^{p + n + 5} P_{n - 4}  - \gamma ^{p + n + 4} P_{n - 3}\,, 
\]
which then yields identity \eqref{eq.oyj86q9}.
\end{proof}
\begin{corollary}
Let $\lambda\in\{p: P_p=0\}$, that is $\lambda\in\{-17, -8, -4, -3, -1\}$. Then the following identities hold for any integer $n$:
\begin{equation}
P_{ - n}  = P_{2n + 3\lambda }  - Q_{n + \lambda } P_{n + 2\lambda }\,,
\end{equation}
\begin{equation}
P_{n + \lambda } Q_{ - n}  = P_{2n + \lambda } Q_n  - P_{3n + \lambda }
\end{equation}
\begin{equation}
Q_{ - n}  = Q_n P_n  - Q_{n - 1} P_{n - 2}  - P_{2n - 2}\,,
\end{equation}
\begin{equation}
P_{\lambda  + 1} P_{ - n}  = P_{\lambda  + n - 4} P_{n - 7}  - P_{\lambda  + n - 3} P_{n - 8}
\end{equation}
and
\begin{equation}
P_{\lambda  - 1} P_{ - n}  = P_{\lambda  + n - 3} P_{n - 7}  - P_{\lambda  + n - 4} P_{n - 6}\,.
\end{equation}
\end{corollary}
\begin{lemma}\label{lem.ytluwst}
The following identities hold for integers $r$, $s$ and $t$, $s\ne 0$:
\begin{equation}\label{eq.jkdhfys}
\left((\alpha^r+\beta^r)\gamma^t\right)_{\gamma^2}=Q_rP_{t-4}-P_{r+t-4}
\end{equation}
and
\begin{equation}\label{eq.df8mpe3}
\left( {\frac{{\alpha ^r  - \beta ^r }}{{\alpha ^s  - \beta ^s }}\gamma ^t } \right)_{\gamma ^2 }  = \frac{{\left| {\begin{array}{*{20}c}
   {P_{t - 3} P_{r - 4}  - P_{t - 4} P_{r - 3} } & {P_{s - 4} } & 0  \\
   {P_{t - 2} P_{r - 4}  - P_{t - 3} P_{r - 3} } & { - P_{s - 3} } & {P_{s - 4} }  \\
   {P_{t - 4} P_{r - 4}  - P_{t - 5} P_{r - 3} } & 0 & { - P_{s - 3} }  \\
\end{array}} \right|}}{{\left| {\begin{array}{*{20}c}
   { - P_{s - 3} } & {P_{s - 4} } & 0  \\
   {P_{s - 4} } & { - P_{s - 3} } & {P_{s - 4} }  \\
   {P_{s - 4} } & 0 & { - P_{s - 3} }  \\
\end{array}} \right|}}\,.
\end{equation}

\end{lemma}
\begin{proof}
Using identity \eqref{eq.gtb313p}, we have
\[
(\alpha ^r  + \beta ^r )\gamma ^t  = (Q_r  - \gamma ^r )\gamma ^t  = Q_r \gamma ^t  - \gamma ^{r + t}\,,
\]
from which identity \eqref{eq.jkdhfys} follows when we use properties P1c and P10.

\medskip

By \eqref{eq.rn0hi2p}, 
\begin{equation}
\frac{{\alpha ^r  - \beta ^r }}{{\alpha ^s  - \beta ^s }}\gamma ^t  = \frac{{\alpha ^r  - \beta ^r }}{{\alpha  - \beta }}\frac{{\alpha  - \beta }}{{\alpha ^s  - \beta ^s }}\gamma ^t  = \frac{{\gamma ^{t + 1} P_{r - 4}  - \gamma ^t P_{r - 3} }}{{\gamma P_{s - 4}  - P_{s - 3} }}\,,
\end{equation}
from which identity \eqref{eq.df8mpe3} follows upon application of properties P1c and P10.
\end{proof}
\begin{lemma}
Let $a$ and $b$ be rational numbers and $f$ and $g$ functions of $\alpha$. The $\alpha^j$ components $\left(\cdot\right)_{\alpha^j}$ have the following composition rules:
\begin{equation}
\left(af\right)_{\alpha ^j }  = a\left(f\right)_{\alpha ^j }\,, 
\end{equation}
\begin{equation}
\left(af + bg\right)_{\alpha ^j }  = a\left(f\right)_{\alpha ^j }  + b\left(g\right)_{\alpha ^j }\,,
\end{equation}
\makeatletter
\renewcommand\tagform@[1]{\maketag@@@{\ignorespaces#1\unskip\@@italiccorr}}
\begin{equation}\tag{R1}
\left(fg\right)_{\alpha ^2 }  = \left(f\right)_{\alpha ^0 } \left(g\right)_{\alpha ^2 }  + \left(f\right)_{\alpha ^2 } \left(g\right)_{\alpha ^0 }  + \left(f\right)_{\alpha } \left(g\right)_{\alpha }  + \left(f\right)_{\alpha ^2 } \left(g\right)_{\alpha ^2 }\,, 
\end{equation}
\begin{equation}\tag{R2}
\left(fg\right)_\alpha   = \left(f\right)_{\alpha ^0 } \left(g\right)_{\alpha }  + \left(f\right)_{\alpha } \left(g\right)_{\alpha ^0 }  + \left(f\right)_{\alpha ^2 } \left(g\right)_{\alpha }  + \left(f\right)_{\alpha } \left(g\right)_{\alpha ^2 }  + \left(f\right)_{\alpha ^2 } \left(g\right)_{\alpha ^2 } 
\end{equation}
and
\begin{equation}\tag{R3}
(fg)_{\alpha ^0 }  = (f)_{\alpha ^0 } (g)_{\alpha ^0 }  + (f)_{\alpha } (g)_{\alpha ^2 }  + (f)_{\alpha ^2 } (g)_{\alpha }\,. 
\end{equation}
\end{lemma}
\makeatother
\begin{theorem}
The following identity holds for integers $m$ and $n$:
\[
P_{m + n}  = P_m P_{n - 5}  + P_{m + 1} P_{n - 3}  + P_{m + 2} P_{n - 4}\,. 
\]

\end{theorem}
\begin{proof}
Set $f=\alpha^m$ and $g=\alpha^n$ in R3, use P9 and note also that $\alpha^m\alpha^n=\alpha^{m+n}$.
\end{proof}
\section{Summation identities}
\subsection{Summation identities not involving binomial coefficients}
\subsubsection{Sums of Padovan and Perrin numbers with subscripts in arithmetic progression}
Setting $x=\alpha^p$ in the geometric sum identity
\begin{equation}\label{eq.cmyl19p}
\sum_{j=0}^nx^j=\frac{x^{n+1}-1}{x-1}
\end{equation}
and multiplying through by $\alpha^{q+4}$ gives
\begin{equation}\label{eq.fqua2xh}
\sum_{j=0}^n\alpha^{pj+q+4}=\frac{\alpha^{pn+p+q+4}-\alpha^{q+4}}{\alpha^p-1}\,.
\end{equation}
Thus, we have
\begin{equation}
\sum_{j = 0}^n {\left(\alpha ^{pj + q + 4} \right)_{\alpha^2}}  = \left(\frac{{(P_{pn + p + q}  - P_q )a^2  + (P_{pn + p + q + 1}  - P_{q + 1} )\alpha + P_{pn + p + q - 1}  - P_{q - 1} }}{{P_{p - 4} \alpha^2  + P_{p - 3} \alpha + P_{p - 5}  - 1}}\right)_{\alpha^2}\,;
\end{equation}
and hence, by P10,
\begin{equation}
\sum_{j = 0}^n {P_{pj + q} }  = \frac{{\left| {\begin{array}{*{20}c}
   {P_{pn + p + q}  - P_q } & {P_{p - 3} } & {P_{p - 4} }  \\
   {P_{pn + p + q + 1}  - P_{q + 1} } & {P_{p - 2}  - 1} & {P_{p - 3} }  \\
   {P_{pn + p + q - 1}  - P_{q - 1} } & {P_{p - 4} } & {P_{p - 5}  - 1}  \\
\end{array}} \right|}}{{\left| {\begin{array}{*{20}c}
   {P_{p - 2}  - 1} & {P_{p - 3} } & {P_{p - 4} }  \\
   {P_{p - 1} } & {P_{p - 2}  - 1} & {P_{p - 3} }  \\
   {P_{p - 3} } & {P_{p - 4} } & {P_{p - 5}  - 1}  \\
\end{array}} \right|}}\,.
\end{equation}
Using \eqref{eq.cmyl19p} and P1d, a similar calculation for the Perrin numbers yields
\begin{equation}
\sum_{j = 0}^n {Q_{pj + q} }  = \frac{{\left| {\begin{array}{*{20}c}
   {Q_{pn + p + q}  - Q_q } & {P_{p - 3} } & {P_{p - 4} }  \\
   {Q_{pn + p + q + 1}  - Q_{q + 1} } & {P_{p - 2}  - 1} & {P_{p - 3} }  \\
   {Q_{pn + p + q - 1}  - Q_{q - 1} } & {P_{p - 4} } & {P_{p - 5}  - 1}  \\
\end{array}} \right|}}{{\left| {\begin{array}{*{20}c}
   {P_{p - 2}  - 1} & {P_{p - 3} } & {P_{p - 4} }  \\
   {P_{p - 1} } & {P_{p - 2}  - 1} & {P_{p - 3} }  \\
   {P_{p - 3} } & {P_{p - 4} } & {P_{p - 5}  - 1}  \\
\end{array}} \right|}}\,.
\end{equation}
In particular we have
\begin{equation}
\sum_{j = 0}^n {P_{j + q} }  = P_{n + q + 5}  - P_{q + 4}
\end{equation}
and
\begin{equation}
\sum_{j = 0}^n {Q_{j + q} }  = Q_{n + q + 5}  - Q_{q + 4}\,.
\end{equation}

\subsubsection{Generating functions of Padovan and Perrin numbers with indices in arithmetic progression}
Setting $x=y\alpha^p$ in the identity
\begin{equation}
\sum_{j = 0}^\infty  {x^j }  = \frac{1}{{1 - x}}
\end{equation}
and multiplying through by $\alpha^{q+4}$ gives
\begin{equation}
\sum_{j = 0}^\infty  {\alpha ^{pj + q + 4} y^j }  = \frac{{\alpha ^{q+4} }}{{1 - \alpha ^p y}}  = \frac{{P_q \alpha ^2  + P_{q + 1} \alpha  + P_{q - 1} }}{{ - yP_{p - 4} \alpha ^2  - yP_{p - 3} \alpha  - yP_{p - 5}  + 1}}\,.
\end{equation}
Thus, by property P10, we have
\begin{equation}\label{eq.pqhgv3t}
\sum_{j = 0}^\infty  {P_{pj + q} y^j }  = \frac{{\left| {\begin{array}{*{20}c}
   {P_q } & { - yP_{p - 3} } & { - yP_{p - 4} }  \\
   {P_{q + 1} } & { - yP_{p - 2}  + 1} & { - yP_{p - 3} }  \\
   {P_{q - 1} } & { - yP_{p - 4} } & { - yP_{p - 5}  + 1}  \\
\end{array}} \right|}}{{\left| {\begin{array}{*{20}c}
   { - yP_{p - 2}  + 1} & { - yP_{p - 3} } & { - yP_{p - 4} }  \\
   { - yP_{p - 1} } & { - yP_{p - 2}  + 1} & { - yP_{p - 3} }  \\
   { - yP_{p - 3} } & { - yP_{p - 4} } & { - yP_{p - 5}  + 1}  \\
\end{array}} \right|}}\,.
\end{equation}
Similarly,
\begin{equation}\label{eq.e90lqpl}
\sum_{j = 0}^\infty  {Q_{pj + q} y^j }  = \frac{{\left| {\begin{array}{*{20}c}
   {Q_q } & { - yP_{p - 3} } & { - yP_{p - 4} }  \\
   {Q_{q + 1} } & { - yP_{p - 2}  + 1} & { - yP_{p - 3} }  \\
   {Q_{q - 1} } & { - yP_{p - 4} } & { - yP_{p - 5}  + 1}  \\
\end{array}} \right|}}{{\left| {\begin{array}{*{20}c}
   { - yP_{p - 2}  + 1} & { - yP_{p - 3} } & { - yP_{p - 4} }  \\
   { - yP_{p - 1} } & { - yP_{p - 2}  + 1} & { - yP_{p - 3} }  \\
   { - yP_{p - 3} } & { - yP_{p - 4} } & { - yP_{p - 5}  + 1}  \\
\end{array}} \right|}}\,.
\end{equation}
\subsubsection{Exponential generating functions of Padovan and Perrin numbers with indices in arithmetic progression}
\begin{theorem}
The following identities hold for integers $p$ and $q$:
\begin{equation}
\sum_{j = 0}^\infty  {\frac{P_{pj + q}}{j!} y^j }  = \frac{{ - i}}{{\sqrt {23} }}\left| {\begin{array}{*{20}c}
   {\alpha ^{q + 4} e^{\alpha ^p y} } & \alpha  & 1  \\
   {\beta ^{q + 4} e^{\beta ^p y} } & \beta  & 1  \\
   {\gamma ^{q + 4} e^{\gamma ^p y} } & \gamma  & 1  \\
\end{array}} \right|
\end{equation}
and
\begin{equation}\label{eq.e6685ns}
\sum_{j = 0}^\infty  {\frac{{Q_{pj + q} }}{{j!}}y^j }  = \frac{{ - i}}{{\sqrt {23} }}\left| {\begin{array}{*{20}c}
   {(2\alpha ^{q + 2}  + \alpha ^{q - 1} )e^{\alpha ^p y} } & \alpha  & 1  \\
   {(2\beta ^{q + 2}  + \beta ^{q - 1} )e^{\beta ^p y} } & \beta  & 1  \\
   {(2\gamma ^{q + 2}  + \gamma ^{q - 1} )e^{\gamma ^p y} } & \gamma  & 1  \\
\end{array}} \right|\,.
\end{equation}
\end{theorem}
\begin{proof}
By Taylor series expansion and P1 we have
\[
\begin{split}
\alpha ^{q + 4} e^{\alpha ^p y}  &= \sum_{j = 0}^\infty  {\frac{{\alpha ^{pj + q + 4} y^j}}{{j!}}}\\
&= \alpha ^2 \sum_{j = 0}^\infty  {\frac{{P_{pj + q} y^j }}{{j!}}}  + \alpha \sum_{j = 0}^\infty  {\frac{{P_{pj + q + 1} y^j }}{{j!}} + \sum_{j = 0}^\infty  {\frac{{P_{pj + q - 1} y^j }}{{j!}}} }\,.
\end{split}
\]
Thus,
\begin{equation}\label{eq.b4nir1k}
\alpha ^{q + 4} e^{\alpha ^p y}  = \alpha ^2 A(p,q) + \alpha B(p,q) + C(p,q)\,,
\end{equation}
where
\begin{equation}
A(p,q) = \sum_{j = 0}^\infty  {\frac{{P_{pj + q} y^j }}{{j!}}}
\end{equation}
and
\[
B(p,q) = \sum_{j = 0}^\infty  {\frac{{P_{pj + q + 1} y^j }}{{j!}},\quad C(p,q) = \sum_{j = 0}^\infty  {\frac{{P_{pj + q - 1} y^j }}{{j!}}} }\,.
\]
Similar calculations give
\begin{equation}\label{eq.o9xh4ao}
\beta ^{q + 4} e^{\beta ^p y}  = \beta ^2 A(p,q) + \beta B(p,q) + C(p,q)
\end{equation}
and
\begin{equation}\label{eq.sn0fnoc}
\gamma ^{q + 4} e^{\gamma ^p y}  = \gamma ^2 A(p,q) + \gamma B(p,q) + C(p,q)\,.
\end{equation}
Solving \eqref{eq.b4nir1k}, \eqref{eq.o9xh4ao} and \eqref{eq.sn0fnoc} simultaneously for $A$, $B$ and $C$, Crammer's rule gives
\[
A(p,q)=\Delta_A/\Delta\,,
\]
where
\[
\Delta=\left| {\begin{array}{*{20}c}
   \alpha^2 & \alpha  & 1  \\
   \beta^2 & \beta  & 1  \\
   \gamma^2 & \gamma  & 1  \\
\end{array}} \right|=i\sqrt {23}
\]
and
\[
\Delta_A=\left| {\begin{array}{*{20}c}
   {\alpha ^{q + 4} e^{\alpha ^p y} } & \alpha  & 1  \\
   {\beta ^{q + 4} e^{\beta ^p y} } & \beta  & 1  \\
   {\gamma ^{q + 4} e^{\gamma ^p y} } & \gamma  & 1  \\
\end{array}} \right|\,.
\]
The proof of \eqref{eq.e6685ns} is similar. We start with
\[
(2\alpha ^{q + 2}  + \alpha ^{q + 2} )e^{\alpha ^p y}  = \sum_{j = 0}^\infty  {(2\alpha ^{pj + q + 2}  + \alpha ^{pj + q - 1} )\frac{{y^j }}{{j!}}} 
\]
and make use of property P1d.
\end{proof}
\subsubsection{Weighted Padovan and Perrin summation}
Replacing $x$ with $x/y$ in identity \eqref{eq.cmyl19p} gives
\begin{equation}\label{eq.bdoy5wn}
(x - y)\sum_{j = 0}^n {y^{r - j} x^j }  = y^{r - n} x^{n + 1}  - y^{r + 1} \,,
\end{equation}
for integers $r$ and $n$ and arbitrary $x$ and $y$.
\begin{theorem}
The following identities hold for integers $m$, $n$ and $r$, for each set of values of $a$, $b$, $c$, $d$ and $e$ given in the table:
\begin{equation}
\begin{split}
&bf^{n + 1}\sum_{j = 0}^n {a^{n - j} P_{m + d + (m + c)r - 4 + (e - c)j} }\\
&\qquad  =f^{n + 1} P_{(m + c)r + (e - c)n + m + e - 4}  - a^{n + 1} P_{(m + c)(r + 1) - 4}
\end{split}
\end{equation}
and
\begin{equation}
\begin{split}
&bf^{n + 1}\sum_{j = 0}^n {a^{n - j} Q_{m + d + (m + c)r - 4 + (e - c)j} }\\
&\qquad  = f^{n + 1} Q_{(m + c)r + (e - c)n + m + e - 4}  - a^{n + 1} Q_{(m + c)(r + 1) - 4}\,.
\end{split}
\end{equation}
\end{theorem}
\begin{proof}
With identity \eqref{eq.zan2cs6} in mind, set $x=f\alpha^{m+e}$ and $y=a\alpha^{m+c}$ in identity \eqref{eq.bdoy5wn}.
\end{proof}
In particular, we have
\begin{equation}
\begin{split}
&bf^{n + 1}\sum_{j = 0}^n {a^{n - j} P_{m + d + (m + c)n - 4 + (e - c)j} }\\
&\qquad  =f^{n + 1} P_{(m+e)(n+1) - 4}  - a^{n + 1} P_{(m + c)(n + 1) - 4}
\end{split}
\end{equation}
and
\begin{equation}
\begin{split}
&bf^{n + 1}\sum_{j = 0}^n {a^{n - j} Q_{m + d + (m + c)n - 4 + (e - c)j} }\\
&\qquad  =f^{n + 1} Q_{(m+e)(n+1) - 4}  - a^{n + 1} Q_{(m + c)(n + 1) - 4}\,.
\end{split}
\end{equation}Here are explicit examples, with the indicated set of values of $a$, $b$, $c$, $d$, $e$ and $f$ as read from the table.
\begin{equation}
\mbox{Set 1: }\sum_{j = 0}^n {P_{m + (m + 1)r - 5 - 3j} }  = P_{(m + 1)(r + 1) - 4}  - P_{(m + 1)r - 3n + m - 6}\,,
\end{equation}
\begin{equation}
\sum_{j = 0}^n {Q_{m + (m + 1)r - 5 - 3j} }  = Q_{(m + 1)(r + 1) - 4}  - Q_{(m + 1)r - 3n + m - 6}\,,
\end{equation}
\begin{equation}
\mbox{Set 7: }\sum_{j = 0}^n {2^{n - j} P_{m + (m - 1)r - 3 - 5j} }  = 2^{n + 1} P_{(m - 1)(r + 1) - 4}  - P_{(m - 1)r - 5n + m - 10}\,,
\end{equation}
\begin{equation}
\sum_{j = 0}^n {2^{n - j} Q_{m + (m - 1)r - 3 - 5j} }  = 2^{n + 1} Q_{(m - 1)(r + 1) - 4}  - Q_{(m - 1)r - 5n + m - 10}\,,
\end{equation}
\begin{equation}
\mbox{Set 10: }\sum_{j = 0}^n {2^{n - j} P_{m + (m + 2)r - 6 + 3j} }  = P_{(m + 2)r + 3n + m + 1}  - 2^{n + 1} P_{(m + 2)(r + 1) - 4}\,,
\end{equation}
\begin{equation}
\sum_{j = 0}^n {2^{n - j} Q_{m + (m + 2)r - 6 + 3j} }  = Q_{(m + 2)r + 3n + m + 1}  - 2^{n + 1} Q_{(m + 2)(r + 1) - 4}\,.
\end{equation}
Further summation identities can be obtained from the following identities:
\begin{equation}\label{eq.sp0lyz6}
x\sum_{j = 0}^n {y^{r - j} (x + y)^j }  = y^{r - n} (x + y)^{n + 1}  - y^{r + 1}
\end{equation}
and
\begin{equation}\label{eq.k66s9zf}
(x - y)\sum_{j = 0}^n {x^{r - j} y^j }  = x^{r + 1} - x^{r - n} y^{n + 1} \,.
\end{equation}
Identity \eqref{eq.sp0lyz6} is obtained by replacing $x$ with $x+y$ in identity \eqref{eq.bdoy5wn} while identity \eqref{eq.k66s9zf} is obtained by interchanging $x$ and $y$ in identity \eqref{eq.bdoy5wn}.

\subsubsection{Sums of certain products of Padovan and Perrin numbers}
Since
\[
\sum_{j = 0}^n {y^{n - j} x^j }=\sum_{j = 0}^n {y^{j} x^{n - j} }\,,
\]
identity \eqref{eq.bdoy5wn} implies
\begin{equation}\label{eq.jes4tt5}
\frac{1}{2}\sum_{j = 0}^n {(xy)^j (x^{n - 2j}  + y^{n - 2j} )}  = \frac{{x^{n + 1}  - y^{n + 1} }}{{x - y}}\,.
\end{equation}
\begin{theorem}
The following identities hold for integers $p$, $q$ and $n$:
\begin{equation}
\begin{split}
\sum_{j = 0}^n {P_{q - pj} Q_{p(n - 2j)} }  &= \frac{{\left| {\begin{array}{*{20}c}
   {P_{pn + 3p + q}  - P_{q - 2pn} } & {P_{3p - 3} } & {P_{3p - 4} }  \\
   {P_{pn + 3p + q + 1}  - P_{q - 2pn + 1} } & {P_{3p - 2}  - 1} & {P_{3p - 3} }  \\
   {P_{pn + 3p + q - 1}  - P_{q - 2pn - 1} } & {P_{3p - 4} } & {P_{3p - 5}  - 1}  \\
\end{array}} \right|}}{{\left| {\begin{array}{*{20}c}
   {P_{3p - 2}  - 1} & {P_{3p - 3} } & {P_{3p - 4} }  \\
   {P_{3p - 1} } & {P_{3p - 2}  - 1} & {P_{3p - 3} }  \\
   {P_{3p - 3} } & {P_{3p - 4} } & {P_{3p - 5}  - 1}  \\
\end{array}} \right|}}\\
&\qquad\qquad +\frac{{2\left| {\begin{array}{*{20}c}
   {P_{q + 1} P_{pn + p - 4}  - P_q P_{pn + p - 3} } & {P_{p - 4} } & 0  \\
   {P_{q + 2} P_{pn + p - 4}  - P_{q + 1} P_{pn + p - 3} } & { - P_{p - 3} } & {P_{p - 4} }  \\
   {P_q P_{pn + p - 4}  - P_{q - 1} P_{pn + p - 3} } & 0 & { - P_{p - 3} }  \\
\end{array}} \right|}}{{\left| {\begin{array}{*{20}c}
   { - P_{p - 3} } & {P_{p - 4} } & 0  \\
   {P_{p - 4} } & { - P_{p - 3} } & {P_{p - 4} }  \\
   {P_{p - 4} } & 0 & { - P_{p - 3} }  \\
\end{array}} \right|}}\,.
\end{split}
\end{equation}
\end{theorem}
\begin{proof}
Set $x=\alpha^p$, $y=\beta^p$ in identity \eqref{eq.jes4tt5} and make use of Lemma \ref{lem.ytluwst}.
\end{proof}
A particular case is
\begin{equation}
\sum_{j=0}^n{P_{q-j}Q_{n-2j}}=P_{n+q+2}-P_{q-2n-1}-2(P_{q+1}P_{n-3}-P_qP_{n-2})\,.
\end{equation}
\subsection{Binomial summation identities}
\subsubsection{Identities from the binomial formula}
With identity \eqref{eq.zan2cs6} in mind; substitute $x=a\alpha^{m+c}$ and $y=b\alpha^{m+d}$ in the binomial formula
\begin{equation}\label{eq.g81atqv}
\sum_{j = 0}^n {\binom njx^j y^{n - j}}  = (x + y)^n\,,
\end{equation}
multiply through by $\alpha^{p+4}$ and make use of properties P1 and P9 to obtain
\begin{theorem}\label{thm.gpl9oyr}
The identity 
\[
\sum_{j=0}^n\binom nja^jb^{n-j}P_{(m+d)n+p+(c-d)j}=f^nP_{(m+e)n+p}\,,
\]
holds for non-negative integer $n$, arbitrary integers $m$ and $p$, and values of $a$, $b$, $c$, $d$ and $e$ as given in the table.
\end{theorem}
The corresponding Perrin version of the identity of Theorem \ref{thm.gpl9oyr} is
\begin{equation}
\sum_{j=0}^n\binom nja^jb^{n-j}Q_{(m+d)n+p+(c-d)j}=f^nQ_{(m+e)n+p}\,.
\end{equation}
Here are some explicit examples from Theorem \ref{thm.gpl9oyr} and sets of values from table
\begin{equation}
\text{Set 1: }\sum_{j = o}^n {( - 1)^j \binom njP_{(m - 1)n + p + 2j} }  = ( - 1)^n P_{(m - 2)n + p} \,,
\end{equation}
\begin{equation}
\sum_{j = o}^n {( - 1)^j \binom njQ_{(m - 1)n + p + 2j} }  = ( - 1)^n Q_{(m - 2)n + p} \,,
\end{equation}
\begin{equation}
\text{Set 4: }\sum_{j = o}^n {\binom njP_{(m - 2)n + p + 4j} }  = P_{(m + 3)n + p} \,,
\end{equation}
\begin{equation}
\sum_{j = o}^n {\binom njQ_{(m - 2)n + p + 4j} }  = Q_{(m + 3)n + p} \,,
\end{equation}
\begin{equation}
\text{Set 7: }\sum_{j = o}^n {( - 1)^j \binom nj2^j P_{(m + 1)n + p - 2j} }  = ( - 1)^n P_{(m - 6)n + p} \,,
\end{equation}
\begin{equation}
\sum_{j = o}^n {( - 1)^j \binom nj2^j Q_{(m + 1)n + p - 2j} }  = ( - 1)^n Q_{(m - 6)n + p} \,,
\end{equation}
\begin{equation}
\text{Set 10: }\sum_{j = o}^n {2^j \binom njP_{(m - 2)n + p + 4j} }  = P_{(m + 5)n + p} \,,
\end{equation}
\begin{equation}
\sum_{j = o}^n {2^j \binom njQ_{(m - 2)n + p + 4j} }  = Q_{(m + 5)n + p} \,,
\end{equation}
\begin{equation}
\text{Set 13: }\sum_{j = o}^n {( - 1)^j \binom njP_{(m - 7)n + p + 14j} }  = ( - 1)^n 4^n P_{(m + 2)n + p} \,,
\end{equation}
\begin{equation}
\sum_{j = o}^n {( - 1)^j \binom njQ_{(m - 7)n + p + 14j} }  = ( - 1)^n 4^n Q_{(m + 2)n + p} \,.
\end{equation}
Many more binomial identities can be derived by making appropriate substitutions in the following variations on the binomial formula:
\begin{equation}\label{eq.xmac84j}
\sum_{j = 0}^n {( - 1)^j \binom nj(x + y)^j y^{n - j}}  = ( - 1)^n x^n\,,
\end{equation}
\begin{equation}\label{eq.y8wpfl4}
\sum_{j = 0}^n {( - 1)^j \binom njx^j (x + y)^{n - j} }  = y^n\,,
\end{equation}
\begin{equation}\label{eq.gt826zn}
\sum_{j = 0}^n {\binom njjx^{j - 1} y^{n - j}}  = n(x + y)^{n - 1}\,,
\end{equation}
\begin{equation}\label{eq.r7394o5}
\sum_{j = 0}^n {( - 1)^j \binom njj(x + y)^{j - 1} y^{n - j}}  = ( - 1)^n nx^{n - 1}
\end{equation}
and
\begin{equation}\label{eq.c0t71uc}
\sum_{j = 1}^n {( - 1)^{j - 1} \binom njx^{j - 1} j(x + y)^{n - j} }  = ny^{n - 1}\,.
\end{equation}
Note that identities \eqref{eq.xmac84j} and \eqref{eq.y8wpfl4} are obtained from identity \eqref{eq.g81atqv} by obvious transformations while identity \eqref{eq.gt826zn} is obtained by differentiating the identity 
\begin{equation}\label{eq.bx3ymai}
\sum_{j = 0}^n {\binom njx^j e^{jz} y^{n - j}}  = (xe^z + y)^n 
\end{equation}
with respect to $z$ and then setting $z$ to zero. 
More generally,
\begin{equation}
\sum_{j = 0}^n {\binom njj^rx^jy^{n - j}}  = \left. {\frac{{d^r }}{{dz^r }}(xe^z + y )^n } \right|_{z = 0}\,.
\end{equation}
Identities \eqref{eq.r7394o5} and \eqref{eq.c0t71uc} are variations on identity \eqref{eq.gt826zn}.
\subsubsection{Identities from Waring identity}
Waring's formula and its dual \cite[Equations (22) and (1)]{gould99} are
\begin{equation}\label{eq.h35j76y}
\sum_{j = 0}^{\left\lfloor {n/2} \right\rfloor } {( - 1)^j \frac{n}{{n - j}}\binom {n-j}j(xy)^j (x + y)^{n - 2j} }  = x^{n}  + y^{n}
\end{equation}
and
\begin{equation}\label{eq.amsa61r}
\sum_{j = 0}^{\left\lfloor {n/2} \right\rfloor } {( - 1)^j \binom {n-j}j(xy)^j (x + y)^{n - 2j} }  = \frac{x^{n + 1}  - y^{n + 1}}{x - y}\,.
\end{equation}
Identity \eqref{eq.h35j76y} holds for positive integer $n$ while identity \eqref{eq.amsa61r} holds for any non-negative integer $n$.
\begin{theorem}\label{thm.dgtxbtl}
The following identities hold for positive integer $n$ and arbitrary integer $p$:
\begin{equation}\label{eq.wnoinfi}
\sum_{j = 0}^{\left\lfloor {n/2} \right\rfloor } {( - 1)^j \frac{n}{{n - j}}\binom {n-j}jP_{p + n - 3j} }  = ( - 1)^n (Q_nP_p-P_{n+p} )\,,
\end{equation}
\begin{equation}\label{eq.tmss2th}
\sum_{j = 0}^{\left\lfloor {n/2} \right\rfloor } {( - 1)^j \binom {n-j}jP_{p + n - 3j} }  = ( - 1)^{n - 1}(P_{p+1}P_{n-3}-P_pP_{n-2}) \,,
\end{equation}
\begin{equation}\label{eq.tnxz92f}
\sum_{j = 0}^{\left\lfloor {n/2} \right\rfloor } {( - 1)^j \frac{n}{{n - j}}\binom{n-j}jP_{p - 4n + 8j} }  = P_{p + n} Q_{2n}   - P_{p + 3n}\,,
\end{equation}
\begin{equation}\label{eq.wp24fl9}
\sum_{j = 0}^{\left\lfloor {n/2} \right\rfloor } {( - 1)^j \binom{n-j}jP_{p - 4n + 8j} }  = P_{p+n} P_{2n-2}  - P_{p + n - 1}P_{2n-1}\,,
\end{equation}
\begin{equation}\label{eq.cl0i5ui}
\sum_{j = 0}^{\left\lfloor {n/2} \right\rfloor } {( - 1)^j \frac{n}{{n - j}}\binom{n-j}jP_{p - 3n + 8j} }  = P_{2n+p} Q_{2n}   - P_{p+4n}
\end{equation}
and
\begin{equation}\label{eq.rgi91fp}
\sum_{j = 0}^{\left\lfloor {n/2} \right\rfloor } {( - 1)^j \binom{n-j}jP_{p - 3n + 8j} }  = P_{2n+p} P_{2n-2}   - P_{2n+p-1}P_{2n-1}\,.
\end{equation}

\end{theorem}
\begin{proof}
Identities \eqref{eq.wnoinfi}, \eqref{eq.tnxz92f} and \eqref{eq.cl0i5ui} are obtained by choosing $(x,y)=(\alpha,\beta)$, $(x,y)=(\alpha/\beta,\beta/\alpha)$ and $(x,y)=(1/\alpha^2,1/\beta^2)$ in identity \eqref{eq.h35j76y}, in turn, and making use of \eqref{eq.nj4qi47}, \eqref{eq.gtb313p}, \eqref{eq.uajx4zx} and \eqref{eq.sclithe}. Identities \eqref{eq.tmss2th}, \eqref{eq.wp24fl9} and \eqref{eq.rgi91fp} are obtained by choosing $(x,y)=(\alpha,\beta)$, $(x,y)=(\alpha/\beta,\beta/\alpha)$ and $(x,y)=(1/\alpha^2,1/\beta^2)$ in identity \eqref{eq.amsa61r}, in turn, and making use of \eqref{eq.nj4qi47}, \eqref{eq.gtb313p}, \eqref{eq.uajx4zx}, \eqref{eq.sclithe} and \eqref{eq.rn0hi2p}.
\end{proof}
\begin{theorem}\label{thm.thdk11q}
The following identities hold for positive integer $n$ and any integer $p$; where values of $a$, $b$, $c$, $d$, $e$ and $f$ are given in the attached table:
\begin{equation}\label{eq.ywavvnk}
\sum_{j = 0}^{\left\lfloor {n/2} \right\rfloor } {( - 1)^j \frac{n}{{n - j}}\binom {n-j}ja^j b^j f^{n - 2j} P_{(m + e)n + p + (c - 2e + d)j} }  = a^n P_{(m + c)n + p}  + b^n P_{(m + d)n + p}
\end{equation}
and
\begin{equation}\label{eq.kvb4it2}
\sum_{j = 0}^{\left\lfloor {n/2} \right\rfloor } {( - 1)^j \frac{n}{{n - j}}\binom {n-j}ja^j b^j f^{n - 2j} Q_{(m + e)n + p + (c - 2e + d)j} }  = a^n Q_{(m + c)n + p}  + b^n Q_{(m + d)n + p}\,.
\end{equation}
\begin{tabular}{lllllllllllllllllllll}
\multicolumn{1}{c}{} & \multicolumn{1}{c}{} & \multicolumn{1}{c}{$a$} & \multicolumn{1}{c}{$b$} & \multicolumn{1}{c}{$c$} & \multicolumn{1}{c}{$d$} & \multicolumn{1}{c}{$e$} & \multicolumn{1}{c}{$f$} & \multicolumn{1}{c}{} & \multicolumn{1}{c}{} & \multicolumn{1}{c}{} & \multicolumn{1}{c}{} & \multicolumn{1}{c}{} & \multicolumn{1}{c}{} & \multicolumn{1}{c}{} & \multicolumn{1}{c}{$a$} & \multicolumn{1}{c}{$b$} & \multicolumn{1}{c}{$c$} & \multicolumn{1}{c}{$d$} & \multicolumn{1}{c}{$e$} & \multicolumn{1}{c}{$f$} \\ 
\cline{1-8}\cline{14-21}
\multicolumn{1}{c}{Set 1:} & \multicolumn{1}{c}{} & \multicolumn{1}{c}{$1$} & \multicolumn{1}{c}{$-1$} & \multicolumn{1}{c}{$1$} & \multicolumn{1}{c}{$-1$} & \multicolumn{1}{c}{$-2$} & \multicolumn{1}{c}{$1$} & \multicolumn{1}{c}{} & \multicolumn{1}{c}{} & \multicolumn{1}{c}{} & \multicolumn{1}{c}{} & \multicolumn{1}{c}{} & \multicolumn{1}{c}{Set 9:} & \multicolumn{1}{c}{} & \multicolumn{1}{c}{$1$} & \multicolumn{1}{c}{$1$} & \multicolumn{1}{c}{$1$} & \multicolumn{1}{c}{$-6$} & \multicolumn{1}{c}{$-1$} & \multicolumn{1}{c}{$2$} \\ 
\multicolumn{1}{c}{Set 2:} & \multicolumn{1}{c}{} & \multicolumn{1}{c}{$1$} & \multicolumn{1}{c}{$-1$} & \multicolumn{1}{c}{$1$} & \multicolumn{1}{c}{$-2$} & \multicolumn{1}{c}{$-1$} & \multicolumn{1}{c}{$1$} & \multicolumn{1}{c}{} & \multicolumn{1}{c}{} & \multicolumn{1}{c}{} & \multicolumn{1}{c}{} & \multicolumn{1}{c}{} & \multicolumn{1}{c}{Set 10:} & \multicolumn{1}{c}{} & \multicolumn{1}{c}{$2$} & \multicolumn{1}{c}{$1$} & \multicolumn{1}{c}{$2$} & \multicolumn{1}{c}{$-2$} & \multicolumn{1}{c}{$5$} & \multicolumn{1}{c}{$1$} \\ 
\multicolumn{1}{c}{Set 3:} & \multicolumn{1}{c}{} & \multicolumn{1}{c}{$1$} & \multicolumn{1}{c}{$1$} & \multicolumn{1}{c}{$-1$} & \multicolumn{1}{c}{$-2$} & \multicolumn{1}{c}{$1$} & \multicolumn{1}{c}{$1$} & \multicolumn{1}{c}{} & \multicolumn{1}{c}{} & \multicolumn{1}{c}{} & \multicolumn{1}{c}{} & \multicolumn{1}{c}{} & \multicolumn{1}{c}{Set 11:} & \multicolumn{1}{c}{} & \multicolumn{1}{c}{$1$} & \multicolumn{1}{c}{$-1$} & \multicolumn{1}{c}{$5$} & \multicolumn{1}{c}{$-2$} & \multicolumn{1}{c}{$2$} & \multicolumn{1}{c}{$2$} \\ 
\multicolumn{1}{c}{Set 4:} & \multicolumn{1}{c}{} & \multicolumn{1}{c}{$1$} & \multicolumn{1}{c}{$1$} & \multicolumn{1}{c}{$2$} & \multicolumn{1}{c}{$-2$} & \multicolumn{1}{c}{$3$} & \multicolumn{1}{c}{$1$} & \multicolumn{1}{c}{} & \multicolumn{1}{c}{} & \multicolumn{1}{c}{} & \multicolumn{1}{c}{} & \multicolumn{1}{c}{} & \multicolumn{1}{c}{Set 12:} & \multicolumn{1}{c}{} & \multicolumn{1}{c}{$1$} & \multicolumn{1}{c}{$-2$} & \multicolumn{1}{c}{$5$} & \multicolumn{1}{c}{$2$} & \multicolumn{1}{c}{$-2$} & \multicolumn{1}{c}{$1$} \\ 
\multicolumn{1}{c}{Set 5:} & \multicolumn{1}{c}{} & \multicolumn{1}{c}{$1$} & \multicolumn{1}{c}{$-1$} & \multicolumn{1}{c}{$3$} & \multicolumn{1}{c}{$2$} & \multicolumn{1}{c}{$-2$} & \multicolumn{1}{c}{$1$} & \multicolumn{1}{c}{} & \multicolumn{1}{c}{} & \multicolumn{1}{c}{} & \multicolumn{1}{c}{} & \multicolumn{1}{c}{} & \multicolumn{1}{c}{Set 13:} & \multicolumn{1}{c}{} & \multicolumn{1}{c}{$1$} & \multicolumn{1}{c}{$-1$} & \multicolumn{1}{c}{$7$} & \multicolumn{1}{c}{$-7$} & \multicolumn{1}{c}{$2$} & \multicolumn{1}{c}{$4$} \\ 
\multicolumn{1}{c}{Set 6:} & \multicolumn{1}{c}{} & \multicolumn{1}{c}{$1$} & \multicolumn{1}{c}{$-1$} & \multicolumn{1}{c}{$3$} & \multicolumn{1}{c}{$-2$} & \multicolumn{1}{c}{$2$} & \multicolumn{1}{c}{$1$} & \multicolumn{1}{c}{} & \multicolumn{1}{c}{} & \multicolumn{1}{c}{} & \multicolumn{1}{c}{} & \multicolumn{1}{c}{} & \multicolumn{1}{c}{Set 14:} & \multicolumn{1}{c}{} & \multicolumn{1}{c}{$1$} & \multicolumn{1}{c}{$-4$} & \multicolumn{1}{c}{$7$} & \multicolumn{1}{c}{$2$} & \multicolumn{1}{c}{$-7$} & \multicolumn{1}{c}{$1$} \\ 
\multicolumn{1}{c}{Set 7:} & \multicolumn{1}{c}{} & \multicolumn{1}{c}{$2$} & \multicolumn{1}{c}{$-1$} & \multicolumn{1}{c}{$-1$} & \multicolumn{1}{c}{$1$} & \multicolumn{1}{c}{$-6$} & \multicolumn{1}{c}{$1$} & \multicolumn{1}{c}{} & \multicolumn{1}{c}{} & \multicolumn{1}{c}{} & \multicolumn{1}{c}{} & \multicolumn{1}{c}{} & \multicolumn{1}{c}{Set 15:} & \multicolumn{1}{c}{} & \multicolumn{1}{c}{$1$} & \multicolumn{1}{c}{$4$} & \multicolumn{1}{c}{$-7$} & \multicolumn{1}{c}{$2$} & \multicolumn{1}{c}{$7$} & \multicolumn{1}{c}{$1$} \\ 
\multicolumn{1}{c}{Set 8:} & \multicolumn{1}{c}{} & \multicolumn{1}{c}{$2$} & \multicolumn{1}{c}{$-1$} & \multicolumn{1}{c}{$-1$} & \multicolumn{1}{c}{$-6$} & \multicolumn{1}{c}{$1$} & \multicolumn{1}{c}{$1$} & \multicolumn{1}{c}{} & \multicolumn{1}{c}{} & \multicolumn{1}{c}{} & \multicolumn{1}{c}{} & \multicolumn{1}{c}{} & \multicolumn{1}{c}{} & \multicolumn{1}{c}{} & \multicolumn{1}{c}{} & \multicolumn{1}{c}{} & \multicolumn{1}{c}{} & \multicolumn{1}{c}{} & \multicolumn{1}{c}{} & \multicolumn{1}{c}{} \\ 
\end{tabular}
\end{theorem}
\begin{proof}
Use $(x,y)=(a\alpha^{m+c},b\alpha^{m+d})$ in \eqref{eq.h35j76y} while taking note of \eqref{eq.zan2cs6}.
\end{proof}
Below we give explicit examples from identities \eqref{eq.ywavvnk} and \eqref{eq.kvb4it2}, using the values of $a$, $b$, $c$, $d$, $e$ and $f$ as given in the indicated set, in each case, as seen from the attached table of Theorem \ref{thm.thdk11q}.
\begin{equation}
\text{Set 1: }\sum_{j = 0}^{\left\lfloor {n/2} \right\rfloor } {\frac{n}{{n - j}}\binom {n-j}jP_{(m - 2)n + p + 4j} }  = P_{(m + 1)n + p}  + (-1)^n P_{(m - 1)n + p}\,,
\end{equation}
\begin{equation}
\text{Set 1: }\sum_{j = 0}^{\left\lfloor {n/2} \right\rfloor } {\frac{n}{{n - j}}\binom {n-j}jQ_{(m - 2)n + p + 4j} }  = Q_{(m + 1)n + p}  + (-1)^n Q_{(m - 1)n + p}\,,
\end{equation}
\begin{equation}
\text{Set 4: }\sum_{j = 0}^{\left\lfloor {n/2} \right\rfloor } {( - 1)^j \frac{n}{{n - j}}\binom {n-j}jP_{(m + 3)n + p - 6j} }  = P_{(m + 2)n + p}  + P_{(m - 2)n + p}\,, 
\end{equation}
\begin{equation}
\text{Set 4: }\sum_{j = 0}^{\left\lfloor {n/2} \right\rfloor } {( - 1)^j \frac{n}{{n - j}}\binom {n-j}jQ_{(m + 3)n + p - 6j} }  = Q_{(m + 2)n + p}  + Q_{(m - 2)n + p}\,, 
\end{equation}
\begin{equation}
\text{Set 7: }\sum_{j = 0}^{\left\lfloor {n/2} \right\rfloor } {\frac{n}{{n - j}}\binom {n-j}j2^j P_{(m - 6)n + p + 12j} }  = 2^n P_{(m - 1)n + p}  + ( - 1)^n P_{(m + 1)n + p}\,, 
\end{equation}
\begin{equation}
\text{Set 7: }\sum_{j = 0}^{\left\lfloor {n/2} \right\rfloor } {\frac{n}{{n - j}}\binom {n-j}j2^j Q_{(m - 6)n + p + 12j} }  = 2^n Q_{(m - 1)n + p}  + ( - 1)^n Q_{(m + 1)n + p}\,,
\end{equation}
\begin{equation}
\text{Set 10: }\sum_{j = 0}^{\left\lfloor {n/2} \right\rfloor } {( - 1)^j \frac{n}{{n - j}}\binom {n-j}j2^j P_{(m + 5)n + p - 10j} }  = 2^n P_{(m + 2)n + p}  + P_{(m - 2)n + p}\,,
\end{equation}
\begin{equation}
\text{Set 10: }\sum_{j = 0}^{\left\lfloor {n/2} \right\rfloor } {( - 1)^j \frac{n}{{n - j}}\binom {n-j}j2^j Q_{(m + 5)n + p - 10j} }  = 2^n Q_{(m + 2)n + p}  + Q_{(m - 2)n + p}\,, 
\end{equation}
\begin{equation}
\text{Set 13: }\sum_{j = 0}^{\left\lfloor {n/2} \right\rfloor } {\frac{n}{{n - j}}\binom {n-j}j2^{2n - 4j} P_{(m + 2)n + p - 4j} }  = P_{(m + 7)n + p}  + ( - 1)^n P_{(m - 7)n + p} 
\end{equation}
and
\begin{equation}
\text{Set 13: }\sum_{j = 0}^{\left\lfloor {n/2} \right\rfloor } {\frac{n}{{n - j}}\binom {n-j}j2^{2n - 4j} Q_{(m + 2)n + p - 4j} }  = Q_{(m + 7)n + p}  + ( - 1)^n Q_{(m - 7)n + p}\,. 
\end{equation}

\subsection{Double binomial summation identities}
\begin{theorem}
The following identities hold for positive integer $n$ and arbitrary integers $m$, $p$ and $q$:
\begin{equation}
\sum_{j = 0}^n {\sum_{k = 0}^j {\binom nj\binom jkP_{p - 4}^k P_{p - 3}^{j - k} P_{p - 5}^{n - j} P_{mn + q + k + j} } }  = P_{(m + p)n + q}\,,
\end{equation}
\begin{equation}
\sum_{j = 0}^n {\sum_{k = 0}^j {\binom nj\binom jkP_{m - 4}^k P_{m - 3}^{j - k} P_{m - 5}^{n - j} P_{pn + q + k + j} } }  = P_{(m + p)n + q}\,,
\end{equation}
\begin{equation}
\sum_{j = 0}^n {\sum_{k = 0}^j {\binom nj\binom jkP_{p - 4}^k P_{p - 3}^{j - k} P_{p - 5}^{n - j} Q_{mn + q + k + j} } }  = Q_{(m + p)n + q}
\end{equation}
and
\begin{equation}
\sum_{j = 0}^n {\sum_{k = 0}^j {\binom nj\binom jkP_{m - 4}^k P_{m - 3}^{j - k} P_{m - 5}^{n - j} Q_{pn + q + k + j} } }  = Q_{(m + p)n + q}\,.
\end{equation}
\end{theorem}
\begin{theorem}
The following identities hold for positive integer $n$ and arbitrary integers $p$ and $q$:
\begin{equation}\label{eq.zu1umgi}
\sum_{j = 0}^{\left\lfloor {n/2} \right\rfloor } {\sum_{k = 0}^{n - 2j} {( - 1)^{j + k} \frac{n}{{n - j}}\binom {n-j}j\binom {n-2j}kQ_p^{n - 2j - k} P_{q - pj + pk} } }  = Q_{pn} P_q  - P_{pn + q}\,,
\end{equation}
\begin{equation}\label{eq.ekkeq9o}
\begin{split}
&\sum_{j = 0}^{\left\lfloor {n/2} \right\rfloor } {\sum_{k = 0}^{n - 2j} {( - 1)^{j + k} \binom {n-j}j\binom {n-2j}kQ_p^{n - 2j - k} P_{q - pj + pk} } }\\
&\qquad\qquad  = \frac{{\left| {\begin{array}{*{20}c}
   {P_{q + 1} P_{pn + p - 4}  - P_q P_{pn + p - 3} } & {P_{p - 4} } & 0  \\
   {P_{q + 2} P_{pn + p - 4}  - P_{q + 1} P_{pn + p - 3} } & { - P_{p - 3} } & {P_{p - 4} }  \\
   {P_q P_{pn + p - 4}  - P_{q - 1} P_{pn + p - 3} } & 0 & { - P_{p - 3} }  \\
\end{array}} \right|}}{{\left| {\begin{array}{*{20}c}
   { - P_{p - 3} } & {P_{p - 4} } & 0  \\
   {P_{p - 4} } & { - P_{p - 3} } & {P_{p - 4} }  \\
   {P_{p - 4} } & 0 & { - P_{p - 3} }  \\
\end{array}} \right|}}\,.
\end{split}
\end{equation}
\begin{equation}\label{eq.tw6t73n}
\sum_{j = 0}^{\left\lfloor {n/2} \right\rfloor } {\sum_{k = 0}^{n - 2j} {( - 1)^{j + k} \frac{n}{{n - j}}\binom {n-j}j\binom {n-2j}kQ_{2p}^k P_{q + 3pn -6pj - 2pk} } }  = (-1)^n(Q_{2pn} P_{pn + q}  - P_{3pn + q})\,,
\end{equation}
\begin{equation}\label{eq.za364v4}
\begin{split}
&\sum_{j = 0}^{\left\lfloor {n/2} \right\rfloor } {\sum_{k = 0}^{n - 2j} {( - 1)^{j + k} \binom {n-j}j\binom {n-2j}kQ_{2p}^k P_{q + 3pn -6pj - 2pk} } }\\
&\qquad  =\frac{{(-1)^n\left| {\begin{array}{*{20}c}
   {P_{pn + q} P_{2pn + 2p - 3}  - P_{pn + q + 1} P_{2pn + 2p - 4} } & {P_{2p - 4} } & 0  \\
   {P_{pn + q + 1} P_{2pn + 2p - 3}  - P_{pn + q + 2} P_{2pn + 2p - 4} } & {P_{2p - 3} } & {P_{2p - 4} }  \\
   {P_{pn + q - 1} P_{2pn + 2p - 3}  - P_{pn + q} P_{2pn + 2p - 4} } & 0 & {P_{2p - 3} }  \\
\end{array}} \right|}}{{\left| {\begin{array}{*{20}c}
   {P_{2p - 3} } & {P_{2p - 4} } & 0  \\
   {P_{2p - 4} } & {P_{2p - 3} } & {P_{2p - 4} }  \\
   {P_{2p - 4} } & 0 & {P_{2p - 3} }  \\
\end{array}} \right|}} \,,
\end{split}
\end{equation}
\begin{equation}\label{eq.yhpq44m}
\sum_{j = 0}^{\left\lfloor {n/2} \right\rfloor } {\sum_{k = 0}^{n - 2j} {( - 1)^{j + k} \frac{n}{{n - j}}\binom {n-j}j\binom {n-2j}kQ_{2p}^k P_{q + 4pn -8pj - 2pk} } }  = (-1)^n(Q_{2pn} P_{2pn + q}  - P_{4pn + q})
\end{equation}
and
\begin{equation}\label{eq.l3j9r2i}
\begin{split}
&\sum_{j = 0}^{\left\lfloor {n/2} \right\rfloor } {\sum_{k = 0}^{n - 2j} {( - 1)^{j + k} \binom {n-j}j\binom {n-2j}kQ_{2p}^k P_{q + 4pn -8pj - 2pk} } } \\ 
&\qquad=\frac{{( - 1)^n \left| {\begin{array}{*{20}c}
   {P_{2pn + q + 1} P_{2pn + 2p - 4}  - P_{2pn + q} P_{2pn + 2p - 3} } & {P_{2p - 4} } & 0  \\
   {P_{2pn + q + 2} P_{2pn + 2p - 4}  - P_{2pn + q + 1} P_{2pn + 2p - 3} } & {P_{2p - 3} } & {P_{2p - 4} }  \\
   {P_{2pn + q} P_{2pn + 2p - 4}  - P_{pn + q - 1} P_{2pn + 2p - 3} } & 0 & {P_{2p - 3} }  \\
\end{array}} \right|}}{{\left| {\begin{array}{*{20}c}
   { - P_{2p - 3} } & {P_{2p - 4} } & 0  \\
   {P_{2p - 4} } & { - P_{2p - 3} } & {P_{2p - 4} }  \\
   {P_{2p - 4} } & 0 & { - P_{2p - 3} }  \\
\end{array}} \right|}}\,.
\end{split}
\end{equation}
\end{theorem}
\begin{proof}
Identities \eqref{eq.zu1umgi}, \eqref{eq.tw6t73n} and \eqref{eq.yhpq44m} are obtained by choosing $(x,y)=(\alpha^p,\beta^p)$, $(x,y)=((\alpha/\beta)^p,(\beta/\alpha)^p)$ and $(x,y)=(1/\alpha^{2p},1/\beta^{2p})$ in identity \eqref{eq.h35j76y}, in turn, and making use of Lemma \ref{lem.ytluwst}. Identities \eqref{eq.ekkeq9o}, \eqref{eq.za364v4} and \eqref{eq.l3j9r2i} are obtained by choosing $(x,y)=(\alpha^p,\beta^p)$, $(x,y)=((\alpha/\beta)^p,(\beta/\alpha)^p)$ and $(x,y)=(1/\alpha^{2p},1/\beta^{2p})$ in identity \eqref{eq.amsa61r}, in turn, and making use of Lemma \ref{lem.ytluwst}.
\end{proof}

\hrule

\noindent 2010 {\it Mathematics Subject Classification}:
Primary 11B37; Secondary 65B10, 11B65.

\noindent \emph{Keywords: }
Padovan sequence, Perrin sequence, generating function, binomial summation.

\hrule

\noindent Concerned with sequences: 
A000931, A001608

%\hrule

\end{document}